\documentclass[12pt]{amsart}
\usepackage{amssymb,ifthen}
\usepackage[T1]{fontenc}
\usepackage[latin9]{inputenc}
\usepackage{mathrsfs}
\usepackage{bm}
\usepackage{amsmath}
\usepackage{amsthm}
\usepackage{amssymb}
\usepackage{stackrel}
\usepackage{graphicx}

\makeatletter

\providecommand{\tabularnewline}{\\}

\theoremstyle{plain}
\newtheorem{lem}{\protect\lemmaname}
\theoremstyle{plain}
\newtheorem{prop}{\protect\propositionname}
\theoremstyle{remark}
\newtheorem{rem}{\protect\remarkname}

\makeatother

\usepackage{babel}
\providecommand{\lemmaname}{Lemma}
\providecommand{\propositionname}{Proposition}
\providecommand{\remarkname}{Remark}

\def\emptystring{}

\renewcommand{\abstract}[6]{\setcounter{equation}{0}
\hbox{}\begin{center}\vskip-2.4truecm
{\small\sc Forum ``Math-for-Industry" 2021}\\[1mm]
{\tiny \bf December 13-16, 2021, Vietnam Institute for Advanced Study in Mathematics}
\end{center}\vspace{1mm}\hrule\vspace{1.2truecm}
\begin{center}
 {\Large\bf #4}\\\vspace{5mm}{\large\bf #2 #1}\\\vspace{2mm}{#3}\\
 \ifthenelse{\equal{#5}{\emptystring}}{}
 {\normalsize \vspace{2mm}(joint work with #5)}
\end{center}
\vspace{7mm} #6}

\begin{document}
\thispagestyle{empty}

%

\abstract{NGUYEN}{Hien Duy}
{School of Mathematics and Physics, University of Queensland, St.
Lucia, Australia \\
Department of Mathematical and Physical Sciences, La Trobe University,
Bundoora, Australia}
{Finite sample inference for generic autoregressive models}{}

Autoregressive models are a class of time series models that are important in both applied and theoretical statistics. Typically, inferential
devices such as confidence sets and hypothesis tests for time series
models require nuanced asymptotic arguments and constructions. We
present a simple alternative to such arguments that allow for the
construction of finite sample valid inferential devices, using a data
splitting approach. We prove the validity of our constructions, as
well as the validity of related sequential inference tools. A set
of simulation studies are presented to demonstrate the applicability
of our methodology.
\section{Introduction}

Let $\left(\Omega,\mathcal{F},\text{Pr}\right)$ be a probability
space, and define a sequence of random variables $\left(X_{t}\left(\omega\right)\right)_{t\in\left[T\right]}$
to be a time series, indexed by $t\in\left[T\right]=\left\{ 1,\dots,T\right\} $,
where $X_{t}=X_{t}\left(\omega\right)\in\mathbb{X}$ for some space
$\mathbb{X}$. We suppose that the time series $\left(X_{t}\right)_{t\in\left[T\right]}$
is order $p\in\mathbb{N}$ autoregressive and parametric, in the sense
that for every $\mathbb{A}\subseteq\mathbb{X}^{p}$,
\[
\text{Pr}\left(\omega:\left(X_{t}\left(\omega\right)\right)_{t\in\left[p\right]}\in\mathbb{A}\right)=\int_{\mathbb{A}}f\left(x_{1},\dots,x_{p};\theta_{0}\right)\text{d}\bm{x}_{1\dots p}\text{,}
\]
and for each $\mathbb{B}\subseteq\mathbb{X}$ and $t>p$,
\[
\text{Pr}\left(\omega:X_{t}\left(\omega\right)\in\mathbb{B}|\mathcal{F}_{t-1}\right)=\int_{\mathbb{B}}f\left(x_{t}|\bm{x}_{t-p\dots t-1};\theta_{0}\right)\text{d}x_{t}\text{.}
\]
Here, $\theta_{0}\in\mathbb{T}$ is a parameter that characterizes
the marginal and conditional probability density functions (PDFs)
\[
f\left(x_{1},\dots,x_{p};\theta_{0}\right)\text{ and }f\left(x_{t}|\bm{x}_{t-p\dots t-1};\theta_{0}\right)\text{, for each }t>p\text{,}
\]
where $\bm{x}_{a\dots b}=\left(x_{a},x_{a+1},\dots,x_{b-1},x_{b}\right)$,
for $a,b\in\mathbb{N}$ such that $a<b$. The symbol $\mathcal{F}_{t}=\sigma\left(X_{1},\dots,X_{t}\right)$
indicates the sigma algebra generated by the random variables $\left(X_{i}\left(\omega\right)\right)_{i\in\left[t\right]}$.
The characterization thus allows us to write the PDF of the
time series $\bm{X}_{T}=\left(X_{t}\right)_{t\in\left[T\right]}$
as
\[
f\left(\bm{x}_{T};\theta_{0}\right)=f\left(x_{1},\dots,x_{p};\theta_{0}\right)\prod_{t=p+1}^{T}f\left(x_{t}|\bm{x}_{t-p\dots t-1};\theta_{0}\right)\text{.}
\]

In this work, we concern ourselves with the problem of drawing inference
about $\theta_{0}$, given that we do not know its value. Specifically,
we are concerned with the construction of $100\left(1-\alpha\right)\%$
confidence sets of the form $\mathcal{\mathscr{C}}^{\alpha}\left(\bm{X}_{T}\right)\subseteq\mathbb{T}$,
where
\[
\text{Pr}_{\theta_{0}}\left(\theta_{0}\in\mathscr{C}^{\alpha}\left(\bm{X}_{T}\right)\right)\ge1-\alpha\text{,}
\]
for any $\alpha\in\left(0,1\right)$. Here, $\text{Pr}_{\theta}$
indicates the probability measure under the assumption that the PDF
of $\bm{X}_{T}$ has form $f\left(\bm{x}_{T};\theta\right)$. We shall
also denote the associated expectation operator by $\text{E}_{\theta}$.

Furthermore, we are interested in testing hypotheses of the form
\begin{equation}
\text{H}_{0}:\theta_{0}\in\mathbb{T}_{0}\text{ versus }\text{H}_{1}:\theta_{0}\in\mathbb{T}_{1}\text{,}\label{eq: hypotheses}
\end{equation}
where $\mathbb{T}_{0},\mathbb{T}_{1}\subseteq\mathbb{T}$ and $\mathbb{T}_{0}\cap\mathbb{T}_{1}=\varnothing$.
Here, we wish to construct valid $P$-values $P_{T}$, where 
\[
\sup_{\theta\in\mathbb{T}_{0}}\text{Pr}_{\theta}\left(P_{T}\le\alpha\right)\le\alpha\text{.}
\]

In order to construct our inference devices, we follow the work of
\cite{Wasserman:2020aa}, who considered the construction of finite
sample valid confidence sets and hypotheses for independent and identically
distributed data (IID), using a data splitting construction with generic
estimators. Due to the lack of reliance on any estimator specific
properties, the authors of \cite{Wasserman:2020aa} refer to their inference
procedures as universal inference (UI).

The UI construction consists of demonstrating that a split data likelihood
ratio construction is an $E$-value, in the sense of \cite{Vovk:2021wm},
and \cite{Koolen2021Log-optimal-any}; i.e., a positive random variable
with expectation less than or equal to 1. The UI construction is extremely
flexible and has been adapted for construction of inferential devices
using composite likelihood ratios \cite{Nguyen2021Universal-infer}
and empirical Bayesian likelihoods \cite{Nguyen2021Finite-sample-i}.
We note that in the simple case of confidence sets for linear first
order autoregressive models, our constructions can be compared to
the finite sample results of \cite{Vovk:2007aa} and \cite[Sec. 4.1]{Bercu:2015aa}.

Besides our constructions of conventional confidence sets and $P$-values,
using the same construction as that of \cite{Wasserman:2020aa}, we
also provide anytime valid confidence set and $P$-value sequences
for sequential estimation from online data, in the spirit of \cite{Johari:2017aa}.
We demonstrate the applicability of some of our constructions via numerical examples.

The paper proceeds as follows. In Section 2, we present our finite
sample confidence set and $P$-value constructions, as well as their
anytime valid counterparts. In Section 3, applications of some of
our constructions are provided via numerical examples. Final remarks
are then provided in Section 4.

\section{Finite sample inference devices}

Let us split $\bm{X}_{T}$ into two contiguous subsequences $\bm{X}_{T}^{1}=\left(X_{1},\dots,X_{T_{1}}\right)$
and $\bm{X}_{T}^{2}=\left(X_{T_{1}+1},\dots,X_{T}\right)$, where
$T_{1}\ge p$. We shall also write $T_{2}=T-T_{1}$. Further, let
$\hat{\Theta}_{T}$ be a generic random estimator, such that
\[
\hat{\Theta}_{T}=\hat{\theta}\left(\bm{X}_{T}^{1}\right)\text{,}
\]
for some function $\hat{\theta}:\mathbb{X}^{T_{1}}\rightarrow\mathbb{T}$,
and define the likelihood ratio statistic
\[
R_{T}\left(\theta\right)=\frac{L\left(\hat{\Theta}_{T};\bm{X}_{T}\right)}{L\left(\theta;\bm{X}_{T}\right)}\text{,}
\]
where
\[
L\left(\theta;\bm{X}_{T}\right)=\prod_{t=T_{1}+1}^{T}f\left(X_{t}|\bm{X}_{t-p\dots t-1};\theta\right)
\]
is the conditional likelihood of $\left[\bm{X}_{T}^{2}|\bm{X}_{T}^{1}\right]$.
\begin{lem}
\label{lem: main}For any $\theta\in\mathbb{T}$, $\mathrm{E}_{\theta}\left[R_{T}\left(\theta\right)\right]\le1$.
\end{lem}
\begin{proof}
Write $\tilde{\bm{X}}_{t-p\dots t-1}=\left(\tilde{X}_{t-p},\dots,\tilde{X}_{t-1}\right)$,
where $\tilde{X}_{t}=X_{t}$, if $t\le T_{1}$, and $\tilde{X}_{t}=x_{t}$,
otherwise. Then
\begin{align*}
 & \text{E}_{\theta}\left[R_{T}\left(\theta\right)\right]\\
= & \text{E}_{\theta}\text{E}_{\theta}\left[R_{T}\left(\theta\right)|\bm{X}_{T}^{1}\right]\\
= & \text{E}_{\theta}\int_{\mathbb{X}^{T_{2}}}\frac{\prod_{t=T_{1}+1}^{T}f\left(x_{t}|\tilde{\bm{X}}_{t-p\dots t-1};\hat{\Theta}_{T}\right)}{\prod_{t=T_{1}+1}^{T}f\left(x_{t}|\tilde{\bm{X}}_{t-p\dots t-1};\theta\right)}\prod_{t=T_{1}+1}^{T}f\left(x_{t}|\tilde{\bm{X}}_{t-p\dots t-1};\theta\right)\text{d}\bm{x}_{T}^{2}\\
= & \text{E}_{\theta}\int_{\mathbb{X}^{T_{2}}}\prod_{t=T_{1}+1}^{T}f\left(x_{t}|\tilde{\bm{X}}_{t-p\dots t-1};\hat{\Theta}_{T}\right)\text{d}\bm{x}_{T}^{2}\\
\stackrel[\text{(i)}]{}{=} & \text{E}_{\theta}\int_{\mathbb{X}}\cdots\int_{\mathbb{X}}f\left(x_{T}|\tilde{\bm{X}}_{T-p\dots T-1};\hat{\Theta}_{T}\right)\text{d}x_{T}\cdots f\left(x_{T_{1}+1}|\tilde{\bm{X}}_{T_{1}-p+1\dots T_{1}};\hat{\Theta}_{T}\right)\text{d}x_{T_{1}+1}\\
\stackrel[\text{(ii)}]{}{=} & \text{E}_{\theta}1=1\text{,}
\end{align*}
where (i) is due to Tonelli's Theorem and (ii) is by definition of
conditional PDFs.
\end{proof}
With Lemma \ref{lem: main} in hand, we can now construct $100\left(1-\alpha\right)\%$
confidence sets of the form
\begin{equation}
\mathscr{C}^{\alpha}\left(\bm{X}_{T}\right)=\left\{ \theta:R_{n}\left(\theta\right)\le1/\alpha\right\} \text{.}\label{eq: Confidence Set}
\end{equation}

\begin{prop}
\label{prop: Confidence}For any $\alpha\in\left(0,1\right)$ and
$\theta_{0}\in\mathbb{T}$,
\[
\mathrm{Pr}_{\theta_{0}}\left(\theta_{0}\in\mathscr{C}^{\alpha}\left(\bm{X}_{T}\right)\right)\ge1-\alpha\text{.}
\]
\end{prop}
\begin{proof}
By Markov's inequality
\[
\text{Pr}_{\theta_{0}}\left(R_{n}\left(\theta_{0}\right)\ge1/\alpha\right)\le\alpha\text{E}_{\theta_{0}}\left[R_{n}\left(\theta_{0}\right)\right]\stackrel[\text{(i)}]{}{=}\alpha\text{,}
\]
where (i) is by Lemma \ref{lem: main}. Then, we complete the proof
by noting that
\begin{align*}
\mathrm{Pr}_{\theta_{0}}\left(\theta_{0}\in\mathscr{C}^{\alpha}\left(\bm{X}_{T}\right)\right) & =1-\text{Pr}_{\theta_{0}}\left(R_{n}\left(\theta_{0}\right)\ge1/\alpha\right)\\
 & \ge1-\alpha\text{.}
\end{align*}
\end{proof}
To test hypotheses of form (\ref{eq: hypotheses}), we require an
additional estimator
\begin{equation}
\tilde{\Theta}_{T}\in\left\{ \tilde{\theta}\in\mathbb{T}:L\left(\tilde{\theta};\bm{X}_{T}\right)\ge L\left(\theta;\bm{X}_{T}\right)\text{, for all }\theta\in\mathbb{T}\right\} \text{.}\label{eq: MLE}
\end{equation}
Then, we may construct the test statistic
\[
S_{T}=R_{T}\left(\tilde{\Theta}_{T}\right)
\]
and its $P$-value $P_{T}=1/S_{T}$.
\begin{prop}
\label{prop: Test}For any $\alpha\in\left(0,1\right)$ and $\mathbb{T}_{0}\subset\mathbb{T}$,
\[
\sup_{\theta\in\mathbb{T}_{0}}\mathrm{Pr}_{\theta}\left(P_{T}\le\alpha\right)\le\alpha\text{.}
\]
\end{prop}
\begin{proof}
For each $\theta\in\mathbb{T}_{0}$, we have
\begin{align*}
\text{E}_{\theta}\left[S_{T}\right] & =\text{E}_{\theta}\left[\frac{L\left(\hat{\Theta}_{T};\bm{X}_{T}\right)}{L\left(\tilde{\Theta}_{T};\bm{X}_{T}\right)}\right]\\
 & \stackrel[\text{(i)}]{}{\le}\text{E}_{\theta}\left[\frac{L\left(\hat{\Theta}_{T};\bm{X}_{T}\right)}{L\left(\theta;\bm{X}_{T}\right)}\right]\\
 & =\text{E}_{\theta}\left[R_{T}\left(\theta\right)\right]\stackrel[\text{(ii)}]{}{=}1\text{,}
\end{align*}
where (i) is by definition (\ref{eq: MLE}) and (ii) is due to Lemma
\ref{lem: main}. Finally, by Markov's inequality, we have
\[
\text{Pr}_{\theta}\left(S_{T}\ge1/\alpha\right)\le\alpha\implies\text{Pr}_{\theta}\left(P_{T}\le\alpha\right)\le\alpha\text{,}
\]
as required.
\end{proof}

\subsection{Anytime valid inference}

Let
\[
M_{T}\left(\theta\right)=\frac{\prod_{t=p+1}^{T}f\left(X_{t}|\bm{X}_{t-p\dots t-1};\hat{\Theta}_{t-1}\right)}{\prod_{t=p+1}^{T}f\left(X_{t}|\bm{X}_{t-p\dots t-1};\theta\right)}\text{,}
\]
for each $T\ge p+1$, and $M_{T}\left(\theta\right)=1$, for each
$T\le p$. We firstly show that $\left(M_{T}\left(\theta\right)\right)_{T\in\mathbb{N}\cup\left\{ 0\right\} }$
is a martingale adapted to the natural filtration $\mathcal{F}_{T}=\sigma\left(X_{1},\dots,X_{T}\right)$.
Here, $\left(\hat{\Theta}_{T}\right)_{T\ge p+1}$ is a non-anticipatory
sequence of estimators of $\theta_{0}$, such that $\hat{\Theta}_{T}$
is dependent only on $\bm{X}_{T}$.
\begin{lem}
\label{lem: Martingale}For each $T\in\mathbb{N}$ and $\theta\in\mathbb{T}$,
$\text{E}_{\theta}\left[M_{T}\left(\theta\right)|\mathcal{F}_{T-1}\right]=M_{T-1}\left(\theta\right)$.
\end{lem}
\begin{proof}
For $T>p+1$,
\begin{align*}
 & \text{E}_{\theta}\left[M_{T}\left(\theta\right)|\mathcal{F}_{T-1}\right]\\
= & \int_{\mathbb{X}}\frac{\prod_{t=p+1}^{T}f\left(\tilde{X}_{t}|\bm{X}_{t-p\dots t-1};\hat{\Theta}_{t-1}\right)}{\prod_{t=p+1}^{T}f\left(\tilde{X}_{t}|\bm{X}_{t-p\dots t-1};\theta\right)}f\left(x_{T}|\bm{X}_{T-p\dots T-1}\right)\text{d}x_{T}\\
= & \frac{\prod_{t=p+1}^{T-1}f\left(X_{t}|\bm{X}_{t-p\dots t-1};\hat{\Theta}_{t-1}\right)}{\prod_{t=p+1}^{T-1}f\left(X_{t}|\bm{X}_{t-p\dots t-1};\theta\right)}\int_{\mathbb{X}}f\left(x_{T}|\bm{X}_{T-p\dots T-1};\hat{\Theta}_{T-1}\right)\text{d}x_{T}\\
\stackrel[\text{(i)}]{}{=} & M_{T-1}\left(\theta\right)\text{.}
\end{align*}
where $\tilde{X}_{T}=x_{T}$ and $\tilde{X}_{t}=X_{t}$, for $t<T$.
Here, (i) is due to the properties of conditional density functions.
For $T\le p+1$, the result holds by definition.
\end{proof}
We now wish to test the hypotheses (\ref{eq: hypotheses}) in a sequential
manner. To do so, we first require an additional sequence of parameter
estimates $\left(\tilde{\Theta}_{T}\right)_{T\ge p+1}$, where
\begin{equation}
\tilde{\Theta}_{T}\in\left\{ \tilde{\theta}\in\mathbb{T}:\prod_{t=p+1}^{T}f\left(X_{t}|\bm{X}_{t-p\dots t-1};\tilde{\theta}\right)\ge\prod_{t=p+1}^{T}f\left(X_{t}|\bm{X}_{t-p\dots t-1};\theta\right)\text{, for all }\theta\in\mathbb{T}\right\} \text{.}\label{eq: MLE martingale}
\end{equation}

Define
\[
N_{T}=M_{T}\left(\tilde{\Theta}_{T}\right)
\]
for $T\ge p+1$ and $N_{T}=1$ for $T\le p$. 
\begin{prop}
\label{prop: anytime test}For each $\alpha\in\left(0,1\right)$ and
$\mathbb{T}_{0}\subset\mathbb{T}$,
\[
\sup_{\theta\in\mathbb{T}_{0}}\mathrm{Pr}_{\theta}\left(\sup_{T\ge0}N_{T}\ge1/\alpha\right)\le\alpha\text{.}
\]
\end{prop}
\begin{proof}
By Lemma \ref{lem: Martingale}, $\left(M_{T}\left(\theta\right)\right)_{T\in\mathbb{N}}$
is a Martingale, and hence by Lemma \ref{lem: Ville}, we have
\[
\text{Pr}_{\theta}\left(\sup_{T\ge0}M_{T}\left(\theta\right)\ge1/\alpha\right)\le\alpha M_{0}\left(\theta\right)\le\alpha\text{.}
\]
Note that for each $T$ and $\theta\in\mathbb{T}_{0}$,
\begin{align*}
N_{T} & =\frac{\prod_{t=p+1}^{T}f\left(X_{t}|\bm{X}_{t-p\dots t-1};\hat{\Theta}_{t-1}\right)}{\prod_{t=p+1}^{T}f\left(X_{t}|\bm{X}_{t-p\dots t-1};\tilde{\Theta}_{T}\right)}\\
 & \stackrel[\text{(i)}]{}{\le}\frac{\prod_{t=p+1}^{T}f\left(X_{t}|\bm{X}_{t-p\dots t-1};\hat{\Theta}_{t-1}\right)}{\prod_{t=p+1}^{T}f\left(X_{t}|\bm{X}_{t-p\dots t-1};\theta\right)}\\
 & =M_{T}\left(\theta\right)\text{,}
\end{align*}
where (i) is due to definition (\ref{eq: MLE martingale}). Thus,
for each $\theta\in\mathbb{T}_{0}$, we have
\[
\text{Pr}_{\theta}\left(\sup_{T\ge0}N_{T}\ge1/\alpha\right)\le\text{Pr}_{\theta}\left(\sup_{T\ge0}M_{T}\left(\theta\right)\ge1/\alpha\right)\le\alpha\text{.}
\]
\end{proof}
We observe that if we define $\bar{P}_{T}=1/N_{T}$, then the sequence
$\left(\bar{P}_{T}\right)_{T\in\mathbb{N}}$ is also valid, in the
sense that
\[
\sup_{\theta\in\mathbb{T}_{0}}\text{Pr}_{\theta}\left(\inf_{T\ge0}\bar{P}_{T}\le\alpha\right)\le\alpha\text{.}
\]
Now, we shall construct sequential confidence sets of the forms
\[
\mathscr{D}_{T}^{\alpha}=\left\{ \theta\in\mathbb{T}:M_{T}\left(\theta\right)\le1/\alpha\right\} \text{.}
\]

\begin{prop}
\label{prop: anytime conf}For any $\alpha\in\left(0,1\right)$ and
$\theta_{0}\in\mathbb{T}$,
\[
\mathrm{Pr}_{\theta_{0}}\left(\theta_{0}\in\mathscr{D}_{T}^{\alpha}\text{, for all }T\in\mathbb{N}\right)\ge1-\alpha\text{.}
\]
\end{prop}
\begin{proof}
Note that $\left\{ \theta_{0}\in\mathscr{D}_{T}^{\alpha}\right\} =\left\{ M_{T}\left(\theta_{0}\right)\le1/\alpha\right\} $
and so 
\[
\mathrm{Pr}_{\theta_{0}}\left(\theta_{0}\in\mathscr{D}_{T}^{\alpha}\text{, for all }T\in\mathbb{N}\right)=\text{Pr}_{\theta_{0}}\left(\sup_{T\ge0}M_{T}\left(\theta_{0}\right)\le1/\alpha\right)\stackrel[\text{(i)}]{}{\ge}1-\alpha\text{,}
\]
where (i) is due to Lemmas \ref{lem: Martingale} and \ref{lem: Ville}.
\end{proof}
Observe that by definition we also have
\[
\mathrm{Pr}_{\theta_{0}}\left(\theta_{0}\in\bar{\mathscr{D}}_{T}^{\alpha}\right)\ge1-\alpha\text{,}
\]
where $\bar{\mathscr{D}}_{T}^{\alpha}=\bigcap_{t=1}^{T}\mathscr{D}_{T}^{\alpha}$,
for each $\alpha\in\left(0,1\right)$ and $T\in\mathbb{N}$.

\section{Numerical examples}

\subsection{\label{subsec: Normal-autoregressive-model}Normal autoregressive
model}

Let $\left(X_{t}\right)_{t\in\mathbb{Z}}$ be a random sequence defined
as
\begin{equation}
X_{t}=\theta_{0}X_{t-1}+E_{t}\text{,}\label{eq: Normal AR1}
\end{equation}
where $\left(E_{t}\right)_{t\in\mathbb{Z}}$ is an IID sequence, with
$E_{t}\sim\text{N}\left(0,1\right)$, for each $t\in\mathbb{Z}$.
We shall construct a confidence interval for $\theta_{0}$ using the
finite sample (FS) procedure.

We take as data $\bm{X}_{T}$, and split the data into two halves
$\bm{X}_{T}^{1}=\left(X_{1},\dots,X_{T_{1}}\right)$ and $\bm{X}_{T}^{2}=\left(X_{T_{1}+1},\dots,X_{T}\right)$,
where $T_{1}=T/2$ (assuming that $T$ is even, for convenience).
Let $\hat{\Theta}_{T}$ be an estimator of $\theta_{0}$ depending
only on $\bm{X}_{T}^{1}$. We use $\hat{\Theta}_{T}$ to construct
the ratio
\begin{align*}
R_{T}\left(\theta\right) & =\frac{\prod_{t=T_{1}+1}^{T}f\left(X_{t}|\bm{X}_{t-p\dots t-1};\hat{\Theta}_{T}\right)}{\prod_{t=T_{1}+1}^{T}f\left(X_{t}|\bm{X}_{t-p\dots t-1};\theta\right)}\\
 & =\frac{\prod_{t=T_{1}+1}^{T}\phi\left(X_{t};\hat{\Theta}_{T}X_{t-1},1\right)}{\prod_{t=T_{1}+1}^{T}\phi\left(X_{t};\theta X_{t-1},1\right)}\\
 & =\exp\left\{ \frac{1}{2}\sum_{t=T_{1}+1}^{T}\left[\left(X_{t}-\hat{\Theta}_{T}X_{t-1}\right)^{2}-\left(X_{t}-\theta X_{t-1}\right)^{2}\right]\right\} \text{,}
\end{align*}
where 
\[
\phi\left(y;\mu,\sigma^{2}\right)=\left(2\pi\sigma^{2}\right)^{-1/2}\exp\left\{ -\frac{1}{2}\frac{\left(y-\mu\right)^{2}}{\sigma^{2}}\right\} \text{,}
\]
is the normal density function with mean $\mu\in\mathbb{R}$ and variance
$\sigma^{2}>0$.

Thus, by Proposition \ref{prop: Confidence}, we obtain $100\left(1-\alpha\right)\%$
confidence intervals (CIs) of form (\ref{eq: Confidence Set}):
\begin{equation}
\mathscr{C}^{\alpha}\left(\bm{X}_{T}\right)=\left\{ \theta\in\mathbb{R}:\frac{1}{2}\sum_{t=T_{1}+1}^{T}\left[\left(X_{t}-\hat{\Theta}_{T}X_{t-1}\right)^{2}-\left(X_{t}-\theta X_{t-1}\right)^{2}\right]\le\log\left(1/\alpha\right)\right\} \text{.}\label{eq: Normal CI}
\end{equation}
Here, we can use the typical least squares (LS) estimator
\begin{align}
\hat{\varTheta}_{T} & =\underset{\theta\in\mathbb{R}}{\arg\min}\sum_{t=2}^{T_{1}}\left(X_{t}-\theta X_{t-1}\right)^{2}=\frac{\sum_{t=2}^{T_{1}}X_{t-1}X_{t}}{\sum_{t=2}^{T_{1}}X_{t-1}^{2}}\text{.}\label{eq: Half LS}
\end{align}

We can compare the performance of CIs of form (\ref{eq: Normal CI})
to the typical asymptotic normal CIs (cf. \cite[Sec. 5.2]{Amemiya1985})
for the LS estimator
\begin{equation}
\Theta_{T}^{\text{LS}}=\frac{\sum_{t=2}^{T}X_{t-1}X_{t}}{\sum_{t=2}^{T}X_{t-1}^{2}}\text{,}\label{eq: Least squares}
\end{equation}
using the distributional limit
\begin{equation}
T^{1/2}\left(\Theta_{T}^{\text{LS}}-\theta_{0}\right)\stackrel{\text{d}}{\longrightarrow}\text{N}\left(0,1-\theta_{0}^{2}\right)\text{.}\label{eq: normal limit AR}
\end{equation}

To assess the relative performance of the FS and LS CIs, we perform
a small simulation study. We simulate $r=1000$ samples of size $T=100$
from model (\ref{eq: Normal AR1}) with $\theta_{0}=0.5$ and construct
$90\%$ CIs. To compare the performances of the CIs, we compute coverage
proportion (CP) (proportion of the $r$ CIs of each type that contain
$\theta_{0}$) and the average length (AL) of the CIs.

We obtain the results $\text{CP}_{\text{FS}}=0.998$ and $\text{CP}_{\text{LS}}=0.895$,
and $\text{AL}_{\text{FS}}=0.643$ and $\text{AL}_{\text{LS}}=0.286$.
We thus observe that both the LS and FS CIs obtain the correct nominal
level of confidence, although the FS CIs are conservative with respect to coverage.
This conservativeness is also reflected in the lengths of the intervals,
where the FS CIs over twice as long as the LS CIs. However, this is
expected given that the FS CIs are constructed only by Markov's inequality
application, whereas the LS CIs makes use of the information geometry
of the normal distribution. Figure \ref{fig: normal CIs} provides
a visualization of 20 pairs of FS and LS CIs from the simulation study.
We observe that in many cases, the FS CIs provide useful inference
regarding the presence of non-zero autocorrelation $\theta_{0}$,
even if the intervals can be larger than necessary.

\begin{figure}
\begin{centering}
\includegraphics[width=12cm]{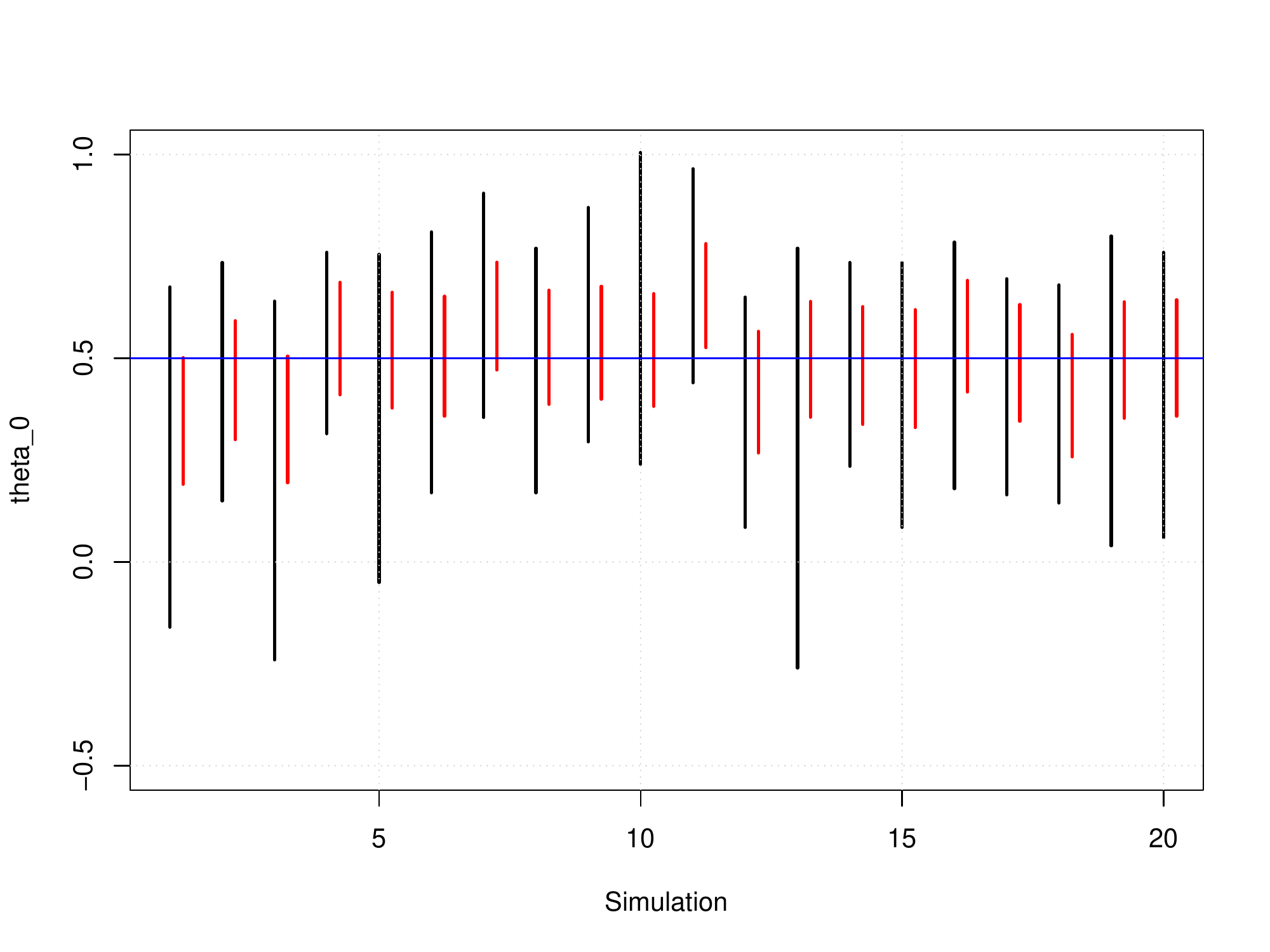}
\par\end{centering}
\caption{\label{fig: normal CIs}A visualization of 20 pairs of 90\% CIs for
$\theta_{0}=0.5$ in the normal autoregressive model. The FS CIs are
colored black and LS CIs are colored red.}
\end{figure}

\subsection{Cauchy autoregressive model}

We now consider model (\ref{eq: Normal AR1}) with $E_{t}\sim\text{Cauchy}\left(0,1\right)$,
which implies that the ratio statistic has form
\begin{align*}
R_{T}\left(\theta\right) & =\frac{\prod_{t=T_{1}+1}^{T}f\left(X_{t}|\bm{X}_{t-p\dots t-1};\hat{\Theta}_{T}\right)}{\prod_{t=T_{1}+1}^{T}f\left(X_{t}|\bm{X}_{t-p\dots t-1};\theta\right)}\\
 & =\frac{\prod_{t=T_{1}+1}^{T}\kappa\left(X_{t}-\hat{\Theta}_{T}X_{t-1}\right)}{\prod_{t=T_{1}+1}^{T}\kappa\left(X_{t}-\theta X_{t-1}\right)}\\
 & =\prod_{T=T_{1}+1}^{T}\frac{1+\left(X_{t}-\theta X_{t-1}\right)^{2}}{1+\left(X_{t}-\hat{\Theta}_{T}X_{t-1}\right)^{2}}\text{,}
\end{align*}
where $\kappa\left(y\right)=\pi^{-1}\left\{ 1/\left(1+y^{2}\right)\right\} $
is the PDF of a the law $\text{Cauchy}\left(0,1\right)$. This implies
a $100\left(1-\alpha\right)\%$ FS CI for $\theta_{0}$ of the form
\[
\mathscr{C}^{\alpha}\left(\bm{X}_{T}\right)=\left\{ \prod_{T=T_{1}+1}^{T}\frac{1+\left(X_{t}-\theta X_{t-1}\right)^{2}}{1+\left(X_{t}-\hat{\Theta}_{T}X_{t-1}\right)^{2}}\le\frac{1}{\alpha}\right\} \text{.}
\]

We again use the LS estimator $\hat{\Theta}_{T}$ to construct the
FS CI and compare our construction to the LS CI using the distributional
limit (\ref{eq: normal limit AR}) as an approximation, since the
Cauchy model does not satisfy the required regularity conditions of
\cite[Sec. 5.2]{Amemiya1985}. The comparison is made via the same
simulation study as described in Section \ref{subsec: Normal-autoregressive-model}.

We obtain the results $\text{CP}_{\text{FS}}=0.995$ and $\text{CP}_{\text{LS}}=0.944$,
and $\text{AL}_{\text{FS}}=0.236$ and $\text{AL}_{\text{LS}}=0.285$.
We notice now that the LS CIs no longer achieve the nominal 90\% confidence
level, and are now also conservative, although not as conservative
as the FS CIs. Interestingly, even though the FS CIs are more conservative,
they are on average shorter than the LS CIs. We observe this via Figure
\ref{fig: Cauchy CIs}, which visualizes 20 pairs of FS and LS CIs
from the simulation study. 

\begin{figure}
\begin{centering}
\includegraphics[width=12cm]{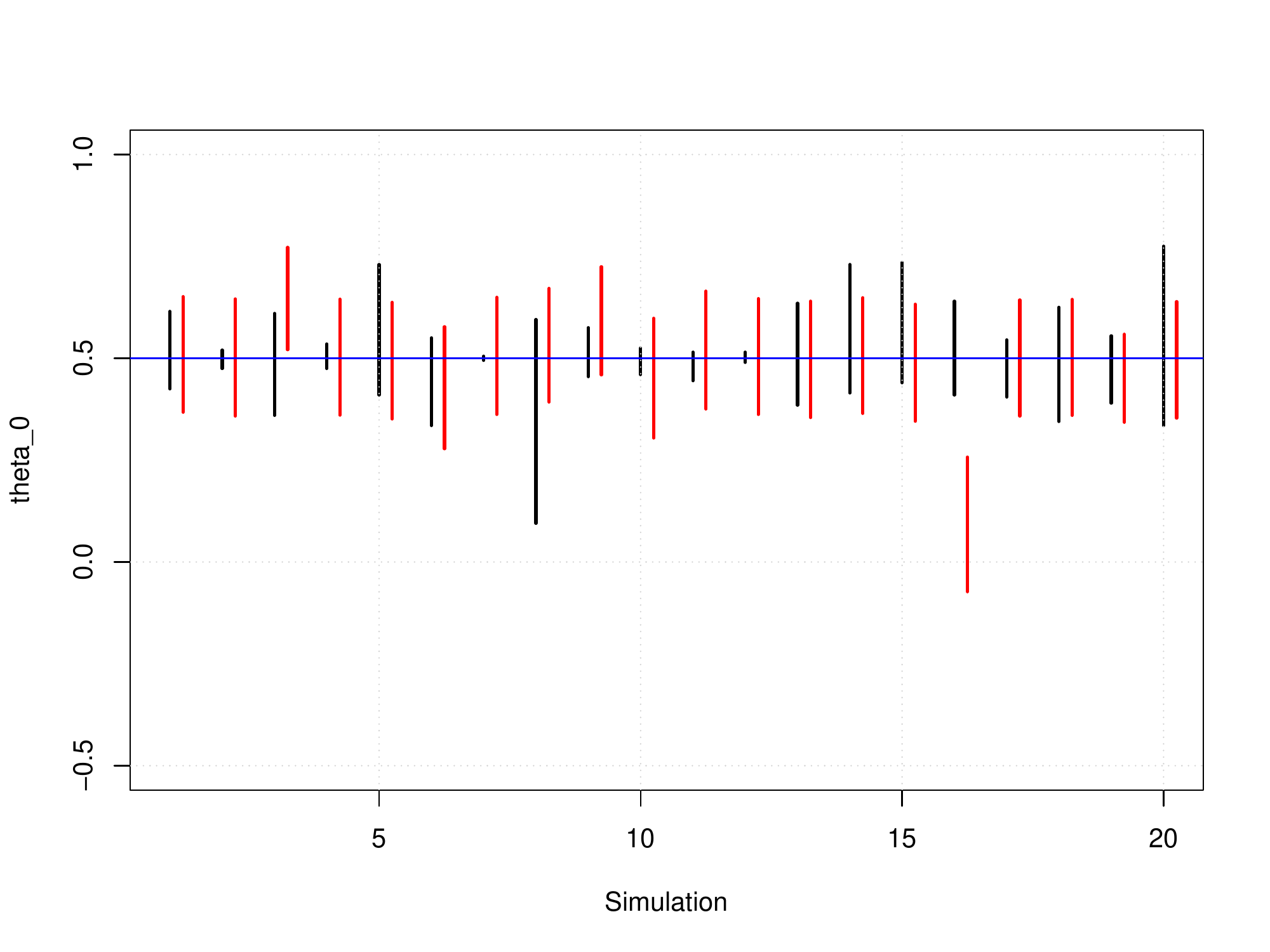}
\par\end{centering}
\caption{\label{fig: Cauchy CIs}A visualization of 20 pairs of 90\% CIs for
$\theta_{0}=0.5$ in the Cauchy autoregressive model. The FS CIs are
colored black and LS CIs are colored red.}
\end{figure}

\section{Unit root test}

We assume again Model \eqref{eq: Normal AR1}, with $E_{t}\sim\text{N}\left(0,1\right)$.
However, we now wish to test the hypotheses
\begin{equation}
\text{H}_{0}:\theta_{0}=1\text{ versus }\text{H}_{1}:\theta_{0}\in\left(-1,1\right)\text{.}\label{eq: Unit Root test}
\end{equation}
This is the classical normal unit root test setting of \cite{Dickey1979Distribution-of},
which is usually tested using the LS estimator (\ref{eq: Least squares})
as the test statistic.

Under the null hypothesis, it is known that the LS estimator has a
non-normal asymptotic distribution that is highly irregular and requires
numerical integration or simulation in order to approximate its quantiles
and density (see, e.g., \cite{Abadir1993The-limiting-di,Evans1981Testing-for-uni,Rao1978Asymptotic-dist}).
However, to perform our FS test, we can simply construct the test
statistic
\begin{equation}
S_{T}=R_{T}\left(1\right)=\exp\left\{ \frac{1}{2}\sum_{t=T_{1}+1}^{T}\left[\left(X_{t}-\hat{\Theta}_{T}X_{t-1}\right)^{2}-\left(X_{t}-X_{t-1}\right)^{2}\right]\right\} \text{,}\label{eq: Unit Root statistic}
\end{equation}
where we use (\ref{eq: Half LS}) for $\hat{\Theta}_{T}$. By Proposition
\ref{prop: Test}, $P_{T}=1/S_{T}$ is a $P$-value, satisfying $\text{Pr}_{\theta_{0}=1}\left(P_{T}\le\alpha\right)\le\alpha$.

We can assess the performance of the FS test based on statistic (\ref{eq: Unit Root statistic})
versus the usual test, based on (\ref{eq: Least squares}), using
the quantiles provided in \cite[Tab. 1]{Abadir1993The-limiting-di}.
We simulate $r=1000$ samples of size $T=1000$ and test (\ref{eq: Unit Root test})
with $\theta_{0}\in\left\{ 0,0.5,0.9,0.09,1\right\} $. We then compare
the asymptotic test to the FS test on the basis of proportion of rejection
(PR) out of the $r$ samples at the $\alpha=0.1$ level of significance.
Our results are presented in Table \ref{tab:Unit-root-test}.

\begin{table}
\caption{\label{tab:Unit-root-test}Unit root test results at the $\alpha=0.1$
level of significance.}

\begin{centering}
\begin{tabular}{ccc}
\hline 
$\theta_{0}$ & FS & Asymptotic\tabularnewline
\hline 
0.00 & 1.000 & 1.000\tabularnewline
0.50 & 1.000 & 1.000\tabularnewline
0.90 & 0.990 & 1.000\tabularnewline
0.95 & 0.904 & 1.000\tabularnewline
1.00 & 0.005 & 0.106\tabularnewline
\hline 
\end{tabular}
\par\end{centering}
\end{table}

From Table \ref{tab:Unit-root-test}, we observe that the FS test
is more conservative than the asymptotic test, as to be expected from
the previous results, along with the Markov's inequality construction.
However, the test does not require knowledge of any special distribution,
and can more easily implemented, as a tradeoff.

\section{Final remarks}
\begin{rem}
The anytime valid inference results of Propositions \ref{prop: anytime test}
and \ref{prop: anytime conf} can be stated in terms
of stopping times of the test and confidence event sequences. This
can be achieved via \cite[Lem. 3]{Howard:2020ab}.
\end{rem}
\begin{rem}
It is noteworthy that the process of splitting the data may be somewhat
arbitrary. However, one alleviate the need of making a choice by averaging
over the results of choices of splits. That is, let $\left(T_{1,i}\right)_{i\in\left[n\right]}$
be a sequence of $n$ values $T_{1,i}\in\left\{ p+1,\dots,T-1\right\} $,
for each $i\in\left[n\right]$, and let $\left(\hat{\Theta}_{T,i}\right)_{i\in\left[n\right]}$
be a sequence of estimators, where $\hat{\Theta}_{T,i}$ depends only
on the data $\left(X_{t}\right)_{t\in\left[T_{1,i}\right]}$. Then,
the averaged ratio statistic
\[
\bar{R}_{T}\left(\theta\right)=\frac{1}{n}\sum_{i=1}^{n}\frac{\prod_{t=T_{1,i}}^{T}f\left(X_{t}|\bm{X}_{t-p\dots t-1};\hat{\Theta}_{T,i}\right)}{\prod_{t=T_{1,i}}^{T}f\left(X_{t}|\bm{X}_{t-p\dots t-1};\theta\right)}
\]
is an $E$-value, in the sense that $\text{E}_{\theta}\left[\bar{R}_{T}\left(\theta\right)\right]\le1$.
Corresponding versions of Propositions \ref{prop: Confidence} and
\ref{prop: Test} then follow.

Here, a choice must still be made regarding the $n$ valued sequence
$\left(T_{1,i}\right)_{i\in\left[n\right]}$. However, one can make
all possible choices, in the sense of taking $\left(T_{1,i}\right)_{i\in\left[n\right]}=\left(p+1,\dots,T-1\right)$.
Then, we would have a ratio statistic in the form 
\[
\bar{R}_{T}\left(\theta\right)=\frac{1}{T-p-1}\sum_{i=p+1}^{T-1}\frac{\prod_{t=i}^{T}f\left(X_{t}|\bm{X}_{t-p\dots t-1};\hat{\Theta}_{T,i}\right)}{\prod_{t=i}^{T}f\left(X_{t}|\bm{X}_{t-p\dots t-1};\theta\right)}\text{,}
\]
where $\left(\hat{\Theta}_{T,i}\right)_{i\in\left\{ p+1,\dots,T-1\right\} }$
is a sequence of estimators with $\hat{\Theta}_{i}$ depending only
on $\left(X_{t}\right)_{t\in\left[i\right]}$, for each $i\in\left\{ p+1,\dots,T-1\right\} $.
This statistic is also an $E$-value and requires no user input regarding
the choice of split. However, it is a much more expensive statistic
than $R_{T}\left(\theta\right)$, since it requires $T-p-1$ estimators
to be computed, whereas $R_{T}\left(\theta\right)$ requires only
one. The user must thus make a tradeoff between computation and user
input.

Since the average of $E$-values is an $E$-value, the same discussion
can be made regarding the choice of estimator $\hat{\Theta}_{T}$.
One can choose different estimators $\hat{\Theta}_{T}$ and average
over the $R_{T}\left(\theta\right)$ statistics corresponding to each
estimator in order to produce a new statistic that remains an $E$-value.
\end{rem}
\begin{rem}
Our text focuses on ratio statistics $R_{T}\left(\theta\right)$ that
are constructed using conditional likelihood objects $L\left(\theta;\bm{X}_{T}\right)$.
However, we may replace the conditional likelihoods with conditional
composite likelihoods or conditional integrated likelihoods, in the
manner of \cite{Nguyen2021Universal-infer} and \cite{Nguyen2021Finite-sample-i},
respectively. This can be useful in situations where the likelihoods
$L\left(\theta;\bm{X}_{T}\right)$ are intractable or difficult to
compute.
\end{rem}

\section*{Appendix}

The following result is often called Ville's Lemma and a proof can
be found in \cite[Thm. 3.9]{Lattimore2020Bandit-Algorith}.
\begin{lem}
\label{lem: Ville}Let $\left(Y_{T}\right)_{T\in\mathbb{N}\cup\left\{ 0\right\} }$
be a non-negative supermartingale, then, for each $\alpha>0$,
\[
\Pr\left(\sup_{T\ge0}Y_{T}\ge1/\alpha\right)\le\alpha\mathrm{E}\left[Y_{0}\right]\text{.}
\]
\end{lem}

\bibliographystyle{plain}
\bibliography{2022_MathsETC}
\end{document}